\documentclass[reqno,10pt]{amsart}
\usepackage{bm,setspace}
\usepackage[vcentermath]{youngtab}
\usepackage{tikz}\usetikzlibrary{decorations.pathreplacing}\usepackage{caption}

\renewcommand{\aa}{{\bm a}}
\newcommand{\bb}{{\bm b}}
\newcommand{\BB}{\mathcal{B}}
\newcommand{\bbox}{\operatorname{box}}
\newcommand{\CC}{\mathbf{C}}
\newcommand{\flush}{\operatorname{flush}}
\newcommand{\non}{\operatorname{non}}

\newcommand{\seg}{\operatorname{seg}}

\newcommand{\wt}{{\rm wt}}
\newcommand{\zz}{{\bm z}}
\newcommand{\ZZ}{\mathbf{Z}}

\theoremstyle{plain}
\newtheorem{thm}{Theorem}[section]

\theoremstyle{definition}
\newtheorem{dfn}[thm]{Definition}
\newtheorem{ex}[thm]{Example}

\newtheorem{remark}[thm]{Remark}
\numberwithin{equation}{section}
\numberwithin{figure}{section}
\numberwithin{table}{section}

\begin{document}
\title{The flush statistic on semistandard Young tableaux}
\author{Ben Salisbury}
\address{Department of Mathematics \\ Central Michigan University \\ Mt. Pleasant, MI 48859}
\email{ben.salisbury@cmich.edu}
\urladdr{http://people.cst.cmich.edu/salis1bt/}
\keywords{Casselman-Shalika formula, crystals, segments, Young tableaux}
\date{\today}
\subjclass[2010]{Primary 05E10; Secondary 17B37}
\begin{abstract}
In this note, a statistic on Young tableaux is defined which encodes data needed for the Casselman-Shalika formula.

\end{abstract}
\maketitle


In joint work with Kyu-Hwan Lee and Phil Lombardo \cite{LLS-A}, the author reinterpreted the Casselman-Shalika formula expression due to Brubaker-Bump-Friedberg \cite{BBF:11-ann,BBF:11-book} and Bump-Nakasuji \cite{BN:10} as a sum over the crystal graph (based on work of Tokuyama \cite{Tok:88}) to a sum over tableaux.  The expression given in the work of Bump and Nakasuji involves taking paths in the  graph of a highest weight crystal from a given vertex to the highest weight vector and decorating the path.  These decorations, called boxing and circling, prescribe contributions at a vertex in the form of Gauss sums (coming from the theory of Weyl group multiple Dirichlet series and Whittaker functions).  The resulting function, formed by summing the contributions over the crystal together with their respective weights, has been coined a {\it Tokuyama function}.  

The benefit to a tableaux description of the Tokuyama function means that, in practice, one no longer needs to compute the entire path to a the highest weight vertex in the crystal graph, which may be very large.  Instead, one can extract the essential data from the content of the tableaux at the vertex and obtain the same function.  This tableaux description was explained in \cite{LLS-A} by using reparametrizations $\aa(T)$ and $\bb(T)$ of the string parametrization obtained directly from data in a tableau $T$, but again required the calculation of a sequence (in our case, two sequences).  It is the goal of this work to interpret these two sequences as statistics on tableaux.

Such a statistic was created in the context of the Gindikin-Karpelevich formula, again based on the work of Brubaker-Bump-Friedberg \cite{BBF:11-ann,BBF:11-book} and Bump-Nakasuji \cite{BN:10}.  This formula, from the context of crystals, may be viewed as the Verma module analogue of the highest weight calculation used in the Casselman-Shalika formula.  In \cite{LS-ABCDG,LS-A}, we were able to recover the path to the highest weight vector using the marginally large tableaux of J.\ Hong and H.\ Lee \cite{HL:08}, which is a certain enlargement of semistandard Young tableaux, and interpret the decorations.  The corresponding statistic was called the \emph{segment} statistic and may be easily read off from a marginally large tableaux.  Outside of type $A_r$, the proof in \cite{LS-ABCDG} relied on the interpretation of the Gindikin-Karpelevich formula as a sum over Lusztig's canonical basis given by H.\ H.\ Kim and K.-H.\ Lee in \cite{KL:11,KL:12} and did not require any decorated paths to the highest weight in the crystal graph.

From the decorated path point of view, the Gindikin-Karpelevich formula only requires the circling rule, and the segment statistic on marginally large tableaux completely encodes the circling data.  Moreover, there exists an embedding from semistandard Young tableaux to marginally large tableaux which preserves the path to the highest weight.  More precisely, if $\BB(\lambda+\rho)$ is the crystal of highest weight $\lambda+\rho$ parametrized by semistandard Young tableaux of shape $\lambda + \rho$, then there is an embedding into the crystal of marginally large tableaux $\BB(\infty)$ such that the circling rule is preserved.  Understanding this embedding leads one to a definition of segments on ordinary semistandard Young tableaux of fixed shape, so that one is only left to understand the boxing rule.  It is the latter problem where the results from \cite{LLS-A} become crucial, as the definition of the sequences defined there lead to a descriptive picture of what the boxing rule means on the tableaux level.

The way to understand the boxing rule again involves the idea of segments in a tableau, but it is how these segments are arranged in the tableau which will encode the boxing rule.  In this note, we define the \emph{flush} statistic on semistandard Young tableaux $T$ in $\BB(\lambda+\rho)$, which, loosely speaking, is the number of segments in $T$ whose left-most box in the segment is in the same column as the left-most box of the subsequent segment in the row beneath it (in English notation for tableaux).  In other words, the number of segments who are flush-left with their neighbor below.  The notions of ``subsequent segment'' and ``neighbor below'' are made precise in Definition \ref{def}(\ref{defflush}) below.  It turns out that $\flush(T)$ is exactly the number of boxed entries of $\aa(T)$ and is equal to the number of boxed entries in the decorated path from $T$ to the highest weight vector of $\BB(\lambda+\rho)$.

It is the hope that the statistics developed here will help shed some new light on the Casselman-Shalika formula outside of type $A_r$.  Currently, there are boxing and circling rules defined and verified in types $B_r$ \cite{FZ:13} and $C_r$ \cite{BBF:11-C,BBF:S}, but only conjectural formulas in types $D_r$ \cite{CG:S} and type $G_2$ \cite{FGG:S}.

\section{Crystals and Tableaux}\label{sec:crystals}
We start by recalling the setup from \cite{LLS-A}.  Let $r\geq 1$ and suppose $\mathfrak{g} = \mathfrak{sl}_{r+1}$ with simple roots $\{ \alpha_1,\dots,\alpha_r\}$, and let $I = \{1,\dots,r\}$. Let $P$ and $P^+$ denote the weight lattice and the set of dominant integral weights, respectively. Denote by $\Phi$ and $\Phi^+$, respectively, the set of roots and the set of positive roots. Let $\{h_1,\dots,h_r\}$ be the set of coroots and define a pairing $\langle \ ,\ \rangle \colon P^\vee\times P \longrightarrow \ZZ$ by $\langle h,\lambda \rangle = \lambda(h)$, where $P^\vee$ is the dual weight lattice. Let $\mathfrak{h} = \CC \otimes_\ZZ P^\vee$ be the Cartan subalgebra, and let $\mathfrak{h}^*$ be its dual. Denote the {\it Weyl vector} by $\rho$; this is the element $\rho\in \mathfrak{h}^*$ defined as
$
\rho = \frac12\sum_{\alpha >0} \alpha = \sum_{i=1}^r \omega_i,
$
where $\omega_i$ is the $i$th fundamental weight.  The set of roots for $\mathfrak{g}$ will be denoted by $\Delta$, while $\Delta^+$ will denote the set of positive roots and $N = \#\Delta^+$.

A {\it $\mathfrak{g}$-crystal} is a set $\BB$ together with maps
$
\widetilde e_i, \widetilde f_i\colon \BB \longrightarrow \BB\sqcup\{0\}$, 
$\varepsilon_i,\varphi_i\colon \BB \longrightarrow \ZZ\sqcup\{-\infty\}$, and 
$\wt\colon \BB \longrightarrow P,
$
such that, for all $b,b'\in \BB$ and $i\in I$, we have $\widetilde f_ib = b'$ if and only if $\widetilde e_ib' = b$, $\wt(\widetilde f_ib) = \wt(b) - \alpha_i$, and $\langle h_i,\wt(b)\rangle = \varphi_i(b) - \varepsilon_i(b)$.  The maps $\widetilde e_i$ (resp.\ $\widetilde f_i$) for $i\in I$ are called the {\it Kashiwara raising operators} (resp.\ {\it Kashiwara lowering operators}).
(For more details, see, for example, \cite{HK:02,Kash:95}.)
To each highest weight representation $V(\lambda)$ of $\mathfrak{g}$, there is an associated highest weight crystal $\BB(\lambda)$ which serves as a combinatorial frame of the representation $V(\lambda)$.  The only fact we will use in this note is that $\BB(\lambda)$ as a set may be realized as the set of semistandard tableaux of shape $\lambda$ over the alphabet $\{1,\dots,r+1\}$ with the usual ordering, where $\lambda = a_1\omega_1 + \cdots + a_r\omega_r$ is identified with the partition having $a_i$ columns of height $i$, for each $1\le i \le r$.

\section{Using the tableaux model}\label{sec:tableaux}

We now recall the definitions and result from \cite{LLS-A}.

\begin{dfn}[\cite{LLS-A}]
Let $\lambda \in P^+$ and $T\in \BB(\lambda+\rho)$ be a tableau.
\begin{enumerate}
\item Define $\aa_{i,j}$ to be the number of $(j+1)$-colored boxes in rows $1$ through $i$ for $1 \leq i \leq j \leq r$, and define the vector $\aa(T) \in \ZZ_{\ge0}^N$ by
\[
\aa(T)=(\aa_{1,1}, \aa_{1,2}, \ldots, \aa_{1,r}; \aa_{2,2}, \ldots , \aa_{2,r}; \ldots ; \aa_{r,r} ).
\]
\item  The number $\bb_{i,j}$ is defined to be  the number of boxes in the $i$th row which have color greater or equal to $j+1$ for $1 \leq i \leq j \leq r$.  Set 
\[
\bb(T) = (\bb_{1,1},\dots,\bb_{1,r};\bb_{2,2},\dots,\bb_{2,r};\cdots;\bb_{r,r}).
\]
\item Write $\lambda+\rho$ as
$
\lambda+\rho = (\ell_1 > \ell_2 > \cdots > \ell_r > \ell_{r+1}= 0),
$
and define $\theta_i = \ell_i- \ell_{i+1}$ for $i=1,\dots,r$.  Let $\theta = (\theta_1,\dots,\theta_r)$.
\end{enumerate}
\end{dfn}

In \cite{LLS-A}, we give a definition of boxing and circling on the entries of $\aa(T)=(\aa_{i,j})$ for $T \in \BB(\lambda + \rho)$ based on the boxing and circling decorations on BZL paths in \cite{BBF:11-ann,BBF:11-book}.
\begin{align}
&\text{Box $\aa_{i,j}$ if $\bb_{i,j} \geq \theta_i + \bb_{i+1,j+1}$.}\tag{B-II}\label{ourbox}\\
&\text{Circle $\aa_{i,j}$ if $\aa_{i,j} = \aa_{i-1,j}$.}\tag{C-II}\label{ourcirc}
\end{align}
Set $\non(T)$ to be the number of entries in $\aa(T)$ which are neither circled nor boxed, and define $\bbox(T)$ to be the number of entries in $\aa(T)$ which are boxed.  Additionally, borrowing the vernacular of the Gelfand-Tsetlin pattern setting of Tokuyama \cite{Tok:88}, we say that $T$ is \emph{strict} if $\aa(T)$ has no entry which is both boxed and circled.  Now define a function $C_{\lambda+\rho,q}$ on $\BB(\lambda+\rho)$ with values in $\ZZ[q^{-1}]$ by
\[
C_{\lambda+\rho}(T;q) =
\left\{\begin{array}{cl}
(-q^{-1})^{\bbox(T)}(1-q^{-1})^{\non(T)} & \text{if $T$ is strict},\\
0 & \text{otherwise}.
\end{array}\right.
\]

\begin{thm}[\cite{LLS-A}]\label{thm:CS-A}
We have
\begin{equation}\label{eq:CS-A}
\zz^\rho\chi_\lambda(\zz)\prod_{\alpha>0} (1-q^{-1}\zz^{-\alpha})
= \sum_{T\in \BB(\lambda+\rho)} C_{\lambda+\rho}(T;q) \zz^{\wt(T)}.
\end{equation}
\end{thm}

The main result of this note is a new statistic on $T\in \BB(\lambda+\rho)$ to compute $C_{\lambda+\rho}(T;q)$ without the need to construct the sequence $\aa(T)$.

\begin{dfn}\label{def}
Let $T\in \BB(\lambda+\rho)$ be a tableau.
\begin{enumerate}
\item\label{defseg} Let $T \in \BB(\lambda+\rho)$ be a tableaux.  We define a {\it $k$-segment} \cite{LS-ABCDG,LS-A} of $T$ (in the $i$th row) to be a maximal consecutive sequence of $k$-boxes in the $i$th row for any $i+1 \le k \le r+1$.  Denote the total number of $k$-segments in $T$ by $\seg(T)$.
\item\label{defflush} Let $1 \le i < k \le r+1$ and suppose $\ell$ is the smallest integer greater than $k$ such that there exists an $\ell$-segment in the $(i+1)$st row of $T$.  A $k$-segment in the $i$th row of $T$ is called {\it flush} if the leftmost box in the $k$-segment and the leftmost box of the $\ell$-segment are in the same column of $T$.  If, however, no such $\ell$ exists, then this $k$-segment is said to be {\it flush} if the number of boxes in the $k$-segment is equal to $\theta_i$.  Denote the number of flush $k$-segments in $T$ by $\flush(T)$.
\end{enumerate}
\end{dfn}

\begin{ex}
Let $r=4$, $\lambda = 2\omega_3$, and
\[
T = \young(112335,23344,3555,5).
\]
It is easy to see that $\seg(T) = 7$ because there is a $2$-segment in the first row, a $3$-segment in both the first and second rows, a $4$-segment in the second row, and a $5$-segment in each of the first, third, and fourth rows.  Moreover, $\flush(T) = 5$ because each $3$-segment and $5$-segment is flush.  In other words, the $2$-segment in the first row and the $4$-segment in the second row are not flush.
\end{ex}

The motivating picture for the definition of flush is given in Figure \ref{fig:flush}, and will useful to keep in mind for the proof of Theorem \ref{algorithm}.

\begin{figure}[ht]
\centering
\begin{tikzpicture}[xscale=1.5,yscale=.65]
\node (i) at (0,1.5) {$i$th row};
\node (i+1) at (0,0.5) {$(i+1)$st row};
\draw[color=white,fill=gray!30] (1,2) -- (3,2) -- (3,1) -- (2,1) -- (2,0) -- (1,0) -- cycle;
\draw[-] (1,2) -- (6,2) -- (6,1) -- (4,1) -- (4,0) -- (1,0) -- cycle;
\draw[-] (4,1) -- (3,1) -- (3,2);
\draw[-] (3,1) -- (2,1) -- (2,0);
\draw[-] (3.1,1.5) -- (5.9,1.5);
\node[fill=white,inner sep=1.7] at (4.5,1.45) {$\bb_{i,k-1}$};
\draw[-] (2.1,0.5) -- (3.9,0.5);
\node[fill=white,inner sep=1.7] at (3,0.45) {$\bb_{i+1,k}$};
\draw [decorate,decoration={brace,amplitude=7pt,mirror},xshift=0.4pt,yshift=-0.4pt](4.,.95) -- (6,.95) node[black,midway,yshift=-0.45cm] {\scriptsize $\theta_i$};
\end{tikzpicture}
\caption{A diagram motivating the definition of flush.}
\label{fig:flush}
\end{figure}
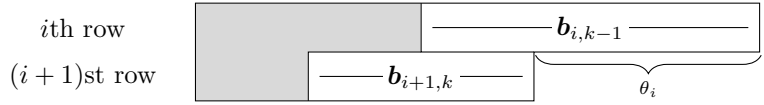

\begin{thm}\label{algorithm}
Let $T\in \BB(\lambda+\rho)$ be a tableau.
\begin{enumerate}
\item\label{circbox} Let $1 \le i < k \le r$.  Suppose the following two conditions hold.
\begin{enumerate}
\item\label{nocircle} There is no $k$-segment in the $i$th row of $T$.
\item\label{nobox} Let $\ell$ be the smallest integer greater than $k$ such that there exist an $\ell$-segment in the $i$th row.  There is no $p$-segment in the $(i+1)$st row, for $k+1\le p \le \ell$, and the $\ell$-segment is flush.\footnote{By convention, if no such $\ell$ exists, then condition (\ref{nobox}) is not satisfied.}
\end{enumerate}
Then $C_{\lambda+\rho}(T;q) = 0$.
\item\label{segflush} If conditions {\upshape(\ref{nocircle})} and {\upshape(\ref{nobox})} are not satisfied, then
\[
C_{\lambda+\rho}(T;q) = (-q^{-1})^{\flush(T)}(1-q^{-1})^{\seg(T)-\flush(T)}.
\]
\end{enumerate}
\end{thm}

\begin{proof}
First note that $\bb_{i,k-1} \ge \bb_{i+1,k} + \theta_i$ implies $\bb_{i,k-1} = \bb_{i+1,k}+\theta_i$ because $T$ is semistandard.

We claim that conditions (\ref{nocircle}) and (\ref{nobox}) are equivalent to $\aa_{i,k-1}$ in $\aa(T)$ being both boxed and circled.  First, there is no $k$-segment in the $i$th row if and only if $\aa_{i,k-1} = \aa_{i-1,k-1}$, which justifies condition (\ref{nocircle}).  It now remains to show that (\ref{nobox}) is equivalent to $\aa_{i,k-1}$ being boxed.  If condition (\ref{nobox}) holds, then, by the definition of $\ell$ and being flush, we have
$
\bb_{i,k-1} = \bb_{i,\ell-1} = \bb_{i+1,\ell} + \theta_i = \bb_{i+1,k} + \theta_i,
$
so $\aa_{i,k-1}$ is boxed.  On the other hand, if $\aa_{i,k-1}$ is boxed and there is no $k$-segment in the $i$th row, then $\bb_{i,k-1} = \bb_{i,\ell-1} = \bb_{i+1,k} + \theta_i$, where $\ell$ is as in condition (\ref{nobox}).  The only way $\bb_{i,\ell-1} = \bb_{i+1,k}+\theta_i$ is if the leftmost box of the $\ell$-segment in the $i$th row and the leftmost box of the $m$-segment in the $(i+1)$st row are in the same column, where $m$ is the smallest integer greater than $k$ such that there exists an $m$-segment in the $(i+1)$st row.  By the semistandardness of $T$, this implies condition (\ref{nobox}) must be satisfied.

To see condition (\ref{segflush}), it follows from Lemma 2.5 and Proposition 2.7 of \cite{LLS-A} that $\seg(T)$ is exactly the number of entries in $\aa(T)$ which are not circled.  Additionally, it follows immediately from the definition that a $k$-segment in the $i$th row is flush if and only if $\bb_{i,k-1} = \bb_{i+1,k}+\theta_i$.  Hence $\bbox(T) = \flush(T)$ and $\non(T) = \seg(T)-\flush(T)$, as required.
\end{proof}

\begin{dfn}
If $T \in \BB(\lambda+\rho)$ is a tableau and if conditions (\ref{nocircle}) and (\ref{nobox}) in Theorem \ref{algorithm} are satisfied for some $1\le i < k \le r$, then we say $T$ \emph{has gaps}.  If no such pair $(i,k)$ satisfies conditions (\ref{nocircle}) and (\ref{nobox}) in Theorem \ref{algorithm}, then we say $T$ is \emph{gapless}.
\end{dfn}

Using the idea of gaps, we may rewrite
\[
C_{\lambda+\rho}(T;q) =
\left\{\begin{array}{cl}
(-q^{-1})^{\flush(T)}(1-q^{-1})^{\seg(T)-\flush(T)} & \text{if $T$ is gapless},\\
0 & \text{if $T$ has gaps},
\end{array}\right.
\]
to get
\begin{equation}\label{CSflush}
\zz^\rho\chi_\lambda(\zz)\prod_{\alpha>0} (1-q^{-1}\zz^{-\alpha})
= \sum_{\substack{T\in \BB(\lambda+\rho) \\ T \,\mathrm{gapless}}} (-q^{-1})^{\flush(T)}(1-q^{-1})^{\seg(T)-\flush(T)} \zz^{\wt(T)}.
\end{equation}

\begin{ex}
Let 
\[
T = \young(11134,2224,34)\ .
\]
There is no $2$-segment in the first row nor is there a $3$-segment in the second row.  However, the $3$-segment in the first row is flush, so $T$ has a gap and $C_{\lambda+\rho}(T;q) =0$.  As a check, we have $\aa(T) = (0,1,1;1,2;3)$ and $\bb(T) = (2,2,1;1,1;1)$.  By \eqref{ourcirc}, $\aa_{1,1} =0$ is circled because the (non-existent) entry $\aa_{0,1} =0$.  Moreover, by \eqref{ourbox}, $\aa_{1,1}$ is boxed because $\bb_{1,1} = 2$ and $\bb_{2,2} + \theta_1 = 1+1$.  
\end{ex}

\begin{ex}
Let 
\[
T = \young(11222334,223334,444)\ .
\]
All possible $k$-segments are included in $T$, so condition (\ref{nocircle}) of Theorem \ref{algorithm} is not satisfied, so $T$ is gapless and $C_{\lambda+\rho}(T;q) \neq 0$.  There is a $2$-segment in the first row, a $3$-segment in both the first and second rows, and a $4$-segment in all three rows.  Thus $\seg(T) = 6$.  Next, the $2$-segment in the first row, the $3$-segment in the first row, and $4$-segment in the last row are each flush, so $\flush(T) = 3$.  Hence $C_{\lambda+\rho}(T;q) = (-q^{-1})^3(1-q^{-1})^3$.  As a check, $\aa(T) = (3,2,1;5,2;5)$, where no entry is circled and $\aa_{1,1}$, $\aa_{1,2}$, and $\aa_{3,3}$ are all boxed, as required.
\end{ex}

\begin{remark}
While the proof above made use of the circling and boxing rules of \cite{BBF:11-book}, the statistics $\seg(T)$ and $\flush(T)$ are intrinsic to the tableaux.  Thus, generalizing these statistics to other Lie algebras may yield an appropriate Casselman-Shalika formula expansion over a crystal graph without the need for circling and boxing rules.   At this moment, such an expansion is important as there are no proven circling and boxing rules in types $D_r$ and $G_2$.  (See \cite{CG:S} for more on the conjecture in type $D_r$ and \cite{FGG:S} for the conjecture in type $G_2$.)
\end{remark}

\section*{Acknowldegements}
The author revisited this work whilst discussing related topics at the ICERM Semester Program on ``Automorphic Forms, Combinatorial Representation Theory and Multiple Dirichlet Series'' during Spring 2013, and then again while preparing a talk at the BIRS workshop entitled ``Whittaker Functions: Number Theory, Geometry, and Physics'' (13w5154) in October 2013, so he would like to thank the various organizers for these opportunities.  The author also valued from informative discussions with Kyu-Hwan Lee and Phil Lombardo.
\bibliography{CS-A}{}
\bibliographystyle{amsplain}
\end{document}